\documentclass[12pt]{amsart}%
\usepackage{graphicx}
\usepackage{amscd, color}
\usepackage{amsmath}
\usepackage{amsfonts}
\usepackage{amssymb}%
\providecommand{\U}[1]{\protect\rule{.1in}{.1in}}
\providecommand{\U}[1]{\protect\rule{.1in}{.1in}}
\providecommand{\U}[1]{\protect\rule{.1in}{.1in}} \textwidth 16.3cm
\theoremstyle{plain}

\newtheorem{theorem}{Theorem}[section]
\newtheorem{proposition}[theorem]{Proposition}
\newtheorem{corollary}[theorem]{Corollary}

\newtheorem{remark}[theorem]{Remark}

\newtheorem{problem}[theorem]{Problem}

\numberwithin{equation}{section}

\begin{document}
\title[Remarks on absolutely summing multilinear operators]{Some remarks on absolutely summing multilinear operators}
\author{A. Thiago Bernardino and Daniel Pellegrino}
\address[D. Pellegrino]{ Departamento de Matem\'{a}tica, Universidade Federal da
Para\'{\i}ba, 58051-900, Jo\~{a}o Pessoa, Brazil\\
[A. Thiago Bernardino] UFRN/CERES - Centro de Ensino Superior do Serid\'{o},
Rua Joaquim Greg\'{o}rio, S/N, 59300-000, Caic\'{o}- RN, Brazil}
\email{dmpellegrino@gmail.com}

\begin{abstract}
This short note has a twofold purpose:

(i) to solve the question that motivates a recent paper of D. Popa on
multilinear variants of Pietsch's composition theorem for absolutely summing
operators. More precisely, we remark that there is a natural perfect extension
of Pietsch's composition theorem to the multilinear and polynomial settings.
This fact was overlooked in the aforementioned paper;

(ii) to investigate extensions of some results of the aforementioned paper for
particular situations, mainly by exploring cotype properties of the spaces involved.

When dealing with (ii) we also prove an useful, albeit simple, result of
independent interest (which is a consequence of recent arguments used in a
recent paper of O. Blasco \textit{et al.}). The result asserts that if $X_{1}$
has cotype $2$ and $1\leq p\leq s\leq2,$ then every absolutely $(s;s,t,...,t)$%
-summing multilinear operator from $X_{1}\times\cdots\times X_{n}$ to $Z$ is
absolutely $(p;p,t,...,t)$-summing, for all $t\geq1$ and all $X_{2}%
,...,X_{n},Z$. In particular, under the same hypotheses, every absolutely
$(s;s,...,s)$-summing multilinear operator from $X_{1}\times\cdots\times
X_{n}$ to $Z$ is absolutely $(p;p,...,p)$-summing. A similar result holds when
$X_{1}$ has cotype greater than $2$ (and obviously, \textit{mutatis mutandis,
}when $X_{1}$ is replaced by $X_{j}$ with $j\neq1$)$.$ These results
generalize previous results of H. Junek \textit{et al.}, G. Botelho \textit{et
al.} and D. Popa.

In the last section we show that a straightforward argument solves partially
another problem from the aforementioned paper of D. Popa.

\end{abstract}
\maketitle

\section{Introduction}

In this note the letters $X_{1},...,X_{n},X,Y,Z$ will always denote Banach
spaces over $\mathbb{K}=\mathbb{R}$ or $\mathbb{C}$ and $X^{\ast}$ represents
the topological dual of $X$.

The concept of absolutely $p$-summing linear operators is due to A. Pietsch
\cite{stu}. If $1\leq q\leq p<\infty,$ we say that a continuous linear
operator $u:X\rightarrow Y$ is absolutely $(p;q)$-summing if $\left(
u(x_{j})\right)  _{j=1}^{\infty}\in\ell_{p}(Y)$ whenever $(x_{j}%
)_{j=1}^{\infty}\in\ell_{q}^{w}(X),$ where $\ell_{q}^{w}(X):=\{(x_{j}%
)_{j=1}^{\infty}\subset X:\sup_{\varphi\in B_{X^{\ast}}}%
{\textstyle\sum\nolimits_{j}}
\left\vert \varphi(x_{j})\right\vert ^{q}<\infty\}.$

The class of absolutely $(p;q)$-summing linear operators from $X$ to $Y$ will
be represented by $\Pi_{p,q}\left(  X;Y\right)  $ and by $\Pi_{p}\left(
X;Y\right)  $ if $p=q.$ From now on the space of all continuous $n$-linear
operators from $X_{1}\times\cdots\times X_{n}$ to $Y$ will be denoted by
$\mathcal{L}(X_{1},...,X_{n};Y).$

If $0<p,q_{1},...,q_{n}<\infty$ and $\frac{1}{p}\leq\frac{1}{q_{1}}%
+\cdots+\frac{1}{q_{n}},$ a multilinear operator $T\in\mathcal{L}%
(X_{1},...,X_{n};Y)$ is absolutely\emph{ }$(p;q_{1},...,q_{n})$-summing if
$(T(x_{j}^{(1)},...,x_{j}^{(n)}))_{j=1}^{\infty}\in\ell_{p}(Y)$ for every
$(x_{j}^{(k)})_{j=1}^{\infty}\in\ell_{q_{k}}^{w}(X_{k}),k=1,...,n.$ In this
case we write $T\in\Pi_{p;q_{1},...,q_{n}}^{n}\left(  X_{1},...,X_{n}%
;Y\right)  $. If $q_{1}=\cdots=q_{n}=q,$ we sometimes write $\Pi_{p;q}%
^{n}\left(  X_{1},...,X_{n};Y\right)  $ instead of $\Pi_{p;q,...,q}^{n}\left(
X_{1},...,X_{n};Y\right)  $ and if $q_{1}=\cdots=q_{n}=q=p$ we simply write
$\Pi_{p}^{n}\left(  X_{1},...,X_{n};Y\right)  $ instead of $\Pi_{p;p}%
^{n}\left(  X_{1},...,X_{n};Y\right)  .$ In the special case in which $p=q/n$
this class has special properties and the operators in $\Pi_{\frac{q}{n}%
;q}^{n}$ are called $q$-dominated operators. Here we will use the notation
$\delta_{q}^{n}=\Pi_{\frac{q}{n};q}^{n}.$

Finally, we recall the class of multiple $(p;q)$-summing multilinear
operators. If $1\leq q\leq p<\infty$, a multilinear operator $T\in
\mathcal{L}(X_{1},...,X_{n};Y)$ is multiple\emph{ }$(p;q)$-summing if
$(T(x_{j_{1}}^{(1)},...,x_{j_{n}}^{(n)}))_{j_{1},..,j_{n}=1}^{\infty}\in
\ell_{p}(Y)$ for every $(x_{j}^{(k)})_{j=1}^{\infty}\in\ell_{q_{k}}^{w}%
(X_{k}),k=1,...,n.$ In this case we write $T\in\Pi_{p;q}^{\text{mult,}%
n}\left(  X_{1},...,X_{n};Y\right)  $ or $\Pi_{p}^{\text{mult,}n}\left(
X_{1},...,X_{n};Y\right)  $ if $p=q$.

For details on the linear theory of absolutely summing operators we refer to
the classical monograph \cite{djt} and for recent developments we refer to
\cite{bp3, PellZ, ku, ku2, sss} and references therein; for the multilinear
theory we refer, for example, to \cite{CD, davidarchiv, PPPP} and references therein.

One important result of the linear theory of absolutely summing linear
operators is Pietsch's composition theorem:

If $p,q\in(1,\infty)$ and $r\in\lbrack1,\infty)$ are such that $\frac{1}%
{r}=\frac{1}{p}+\frac{1}{q}$, then
\begin{equation}
\Pi_{q}\circ\Pi_{p}\subset\Pi_{r}. \label{uty}%
\end{equation}
In a recent paper \cite{popa} this result is investigated in the context of
multilinear mappings. The first question faced in \cite{popa} was to decide
what should be the natural class of absolutely $p$-summing $n$-linear mappings
$\mathcal{I}_{p}^{n}$ such that the analogous result would hold in the
multilinear setting. More precisely the following problem summarize
mathematically the motivation of the paper \cite{popa} (see \cite[Section
1]{popa}):

\begin{problem}
\label{has}If $p,q\in(1,\infty)$ and $r\in\lbrack1,\infty)$ are such that
$\frac{1}{r}=\frac{1}{p}+\frac{1}{q}$, does the inclusion
\begin{equation}
\Pi_{q}\circ\mathcal{I}_{p}^{n}\subset\mathcal{I}_{r}^{n} \label{bein}%
\end{equation}
always hold for all natural numbers $n$ and some natural $n$-linear extension
$\left(  \mathcal{I}_{s}^{n}\right)  _{s=1}^{\infty}$ of $\left(  \Pi
_{s}\right)  _{s=1}^{\infty}?$
\end{problem}

In \cite{popa} it is shown that the inclusion (\ref{bein}) is far from being
true for the class of dominated $n$-linear mappings, i.e., $\mathcal{I}%
_{p}^{n}=\delta_{p}^{n}$ and $\mathcal{I}_{r}^{n}=\delta_{r}^{n}.$ So the
author decided to investigate the case $\mathcal{I}_{p}^{n}=\delta_{p}^{n}$
and $\mathcal{I}_{r}^{n}=\Pi_{r}^{n},$ i.e., the following question was considered:

\begin{problem}
\label{opoi}Let $p,q\in(1,\infty)$ and $r\in\lbrack1,\infty)$ be such that
$\frac{1}{r}=\frac{1}{p}+\frac{1}{q}$. For what natural numbers $n$ the
inclusion
\[
\Pi_{q}\circ\delta_{p}^{n}\subset\Pi_{r}^{n}%
\]
is true?
\end{problem}

Among other interesting results, in \cite[Theorem 4 and Corollary 19]{popa} it
is proved that the above inclusion is valid for all $n$ and $r\in\left[
1,2\right]  :$

\begin{theorem}
\label{op}(\cite{popa})Let $p,q\in(1,\infty)$ and $r\in\lbrack1,2]$ be such
that $\frac{1}{r}=\frac{1}{p}+\frac{1}{q}$. Then
\begin{equation}
\Pi_{q}\circ\delta_{p}^{n}\subset\Pi_{r}^{n} \label{p00}%
\end{equation}
for all natural numbers $n.$
\end{theorem}

In view of Theorem \ref{op} the following problem is posed \cite{popa} (in the
last section we use a very simple remark to solve this problem for all
$n\geq\frac{p}{r}$):

\begin{problem}
\label{novo}Let $p,q\in(1,\infty)$ and $r\in(2,\infty)$ be such that $\frac
{1}{r}=\frac{1}{p}+\frac{1}{q}$. For what natural numbers $n$ the inclusion
$\Pi_{q}\circ\delta_{p}^{n}\subset\Pi_{r}^{n}$ is true?
\end{problem}

We believe that by considering different classes ($p$-dominated and absolutely
$r$-summing $n$-linear operators), the Problem \ref{opoi} becomes a little bit
far from the original motivation (\ref{uty}). But, of course, Problem
\ref{opoi} has its intrinsic mathematical interest and a complete solution
seems to be far from being simple.

It is worth mentioning that the class $\Pi_{r}^{n}$ (although this class had
been broadly explored by several authors and also offers interesting
challenging problems) is usually not considered as a completely adequate
extension of $\Pi_{r},$ since several of the linear properties of $\Pi_{r}$
are not lifted to $\Pi_{r}^{n}$ (this kind of fault of the class $\Pi_{r}^{n}$
- and its polynomial version - was discussed in some recent papers (see, for
example, \cite[page 167]{popa3} and \cite{CDM09, rrr, QM})). Using the
terminology of \cite{CDM09} it can be said that the ideal of absolutely
$r$-summing $n$-homogeneous polynomials (associated to $\Pi_{r}^{n}$) is not
compatible with the linear operator ideal $\Pi_{r}$. For details on operator
ideals we refer to the classical monograph \cite{pp1} and \cite{ddjj}.

The case $r=1$ of Theorem \ref{op} (\cite[Theorem 4]{popa}) deserves some
special attention. Contrary to the case $n=1$, for $n\geq2,$ in many cases,
i.e., for several Banach spaces $X_{1},...,X_{n},Y$, the space $\Pi_{1}%
^{n}(X_{1},...,X_{n};Y)$ coincides with the whole space of continuous
multilinear operators $\mathcal{L}(X_{1},...,X_{n};Y)$ and Theorem \ref{op}
(with $r=1$) becomes useless. For example:

\begin{itemize}
\item For all Banach spaces $X_{1},...,X_{n}$ the folkloric Defant-Voigt
Theorem asserts that%
\begin{equation}
\Pi_{1}^{n}(X_{1},...,X_{n};\mathbb{K})=\mathcal{L}(X_{1},...,X_{n}%
;\mathbb{K}). \label{c1}%
\end{equation}

\item (\cite{irish}) If each $X_{j}$ is a Banach space with cotype $q_{j}$ for
every $j$ and $1\leq\frac{1}{q_{1}}+\cdots+\frac{1}{q_{n}}$, then
\begin{equation}
\Pi_{1}^{n}(X_{1},...,X_{n};Y)=\mathcal{L}(X_{1},...,X_{n};Y) \label{c2}%
\end{equation}
for every Banach space $Y.$
\end{itemize}

It must be said that the paper \cite{popa} also presents several interesting
variants of Theorem \ref{op} (including the case $r=1$), replacing, for
example, $\Pi_{1}^{n}$ by $\Pi_{t;r}^{n}$ for some values of $t,r$ (depending
on $n$, in general).

This short note has two main goals. The first goal is to present the precise
ideal $\mathcal{I}_{p}^{n}$ that solves completely Problem \ref{has}. Our
second goal is to look for stronger variants of Theorem \ref{op}, specially
under certain special cotype assumptions. For example (using a completely
different approach from the one in \cite{popa}), we show that when $X_{j}$ has
cotype $2$ for some $j$ and $Y$ has cotype $2$ then the inclusion%

\[
\Pi_{q}(Y,Z)\circ\delta_{p}^{n}(X_{1},...,X_{n};Y)\subset\Pi_{r}^{n}%
(X_{1},...,X_{n};Z)
\]
from Theorem \ref{op} (with $\frac{1}{p}+\frac{1}{q}=\frac{1}{r}$ and
$r\in\lbrack1,2]$) can be replaced by%
\[
\Pi_{q}(Y;Z)\circ%
{\displaystyle\bigcup\nolimits_{p\geq1}}
\delta_{p}^{n}(X_{1},...,X_{n};Y)\subset\Pi_{r}^{n}(X_{1},...,X_{n};Z)
\]
for all $q\in\lbrack1,\infty),$ all $r\in\lbrack1,2]$ and all Banach spaces
$X_{1},...X_{j-1},X_{j+1},...,X_{n}$,$Z.$

In the last section we give a simple partial answer to Problem \ref{novo} by
showing that the inclusion holds whenever $n\geq\frac{p}{r}$ (in fact we do
not need that $\frac{1}{p}+\frac{1}{q}=\frac{1}{r}.$ This fact is apparently
overlooked in \cite{popa}).

\section{The solution to Problem \ref{has}}

The ideal of absolutely $p$-summing linear operators has various possible
generalizations to multi-ideals: absolutely $p$-summing multilinear operators
(\cite{am, dan}), $p$-dominated multilinear operators (\cite{irish, cg, anais,
davidarchiv}), strongly $p$-summing multilinear operators (\cite{di}),
strongly fully $p$-summing multilinear operators (\cite{port}), multiple
$p$-summing multilinear operators (\cite{collec, pv}), absolutely $p$-summing
multilinear operators by the method of linearization (\cite{note}),
$p$-semi-integral multilinear operators (\cite{CD}) and the composition ideal
generated by the ideal of absolutely $p$-summing linear operators
(\cite{prims}).

Each of these classes has its own properties and shares part of the spirit of
the linear concept of absolutely $p$-summing operators. The richness of the
multilinear theory of absolutely summing operators and multiplicity of
different possible approaches has attracted the attention of several
mathematicians in the last two decades. One of the beautiful features is that
no one of these classes shares all the desired properties of the ideal of
absolutely $p$-summing linear operators and depending on the properties that
one looks for, the \textquotedblleft natural\textquotedblright\ class to be
considered changes. However, it seems to be clear that the most popular
classes until now are the ideals of $p$-dominated multilinear operators and
multiple $p$-summing multilinear operators (but the classes that seem to be
closest to the essence of the linear ideal are, in our opinion, the classes of
strongly $p$-summing multilinear operators and strongly fully (also called
strongly multiple) $p$-summing multilinear operators). For a recent survey on
this subject we refer to \cite{QM}.

In this section we remark that the composition ideal generated by the
absolutely $p$-summing multilinear operators is precisely the class that
completely answers Problem \ref{has}.

If $\mathcal{I}$ is an operator ideal it is always possible to consider the
class%
\[
\mathcal{C}_{\mathcal{I}}^{n}:=\left\{  u\circ A:A\in\mathcal{L}^{n}\text{ and
}u\in\mathcal{I}\right\}  ,
\]
where $\mathcal{L}^{n}$ denotes the class of all continuous $n$-linear
operators between Banach spaces. So, for Banach spaces $X_{1},...,X_{n},Y,Z$,
an $n$-linear operator $T:X_{1}\times\cdots\times X_{n}\rightarrow Y$ belongs
to $\mathcal{C}_{\mathcal{I}}^{n}(X_{1},...,X_{n};Y)$ if and only if there are
a Banach space $Z$, a map $A\in\mathcal{L}(X_{1},...,X_{n};Z)$ and
$v\in\mathcal{I}(Z;Y)$ so that
\[
T(x_{1},...,x_{n})=v\left(  A(x_{1},...,x_{n}\right)  ).
\]
It is well known that $\mathcal{C}_{\mathcal{I}}^{n}$ is an ideal of
$n$-linear mappings (for details see \cite{prims}).\ The case where
$\mathcal{I}=\Pi_{1}$ was investigated in \cite{prims}, where it was shown
that this class lifts, to the multilinear setting, various important features
of the linear ideal, such as a Dvoretzky-Rogers theorem, a Grothendieck
theorem and a Lindenstrauss-Pe\l czy\'{n}ski theorem (three important
cornerstones of the linear theory of absolutely summing operators). The
solution to Problem \ref{has} is now quite simple and the proof is a
straightforward consequence of Pietsch's composition theorem for absolutely
summing linear operators:

\begin{proposition}
If $p,q\in(1,\infty)$ and $r\in\lbrack1,\infty)$ are such that $\frac{1}%
{r}=\frac{1}{p}+\frac{1}{q}$, then
\[
\Pi_{q}\circ\mathcal{C}_{\Pi_{p}}^{n}\subset\mathcal{C}_{\Pi_{r}}^{n}%
\]
for all natural number $n$.
\end{proposition}

\begin{remark}
It is worth mentioning that the polynomial version of the multi-ideal $\left(
\mathcal{C}_{\mathcal{I}}^{n}\right)  _{n=1}^{\infty}$ (which we denote by
$\left(  \mathcal{CP}_{\mathcal{I}}^{n}\right)  _{n=1}^{\infty}$) also solves
the polynomial version of Problem \ref{has}. Moreover, the polynomial ideal
$\left(  \mathcal{CP}_{\mathcal{I}}^{n}\right)  _{n=1}^{\infty}$ is a coherent
sequence and compatible with $\mathcal{I}$ (see \cite{CDM09}), reinforcing
that the composition method is an adequate method for generalizing the ideal
of absolutely summing operators.
\end{remark}

\section{Some remarks related to Problem \ref{opoi}}

Although the very simple solution to Problem \ref{has}, we do think that
Problem \ref{opoi} is interesting and now we investigate how the results from
\cite{popa} can be improved in certain special situations. In view of the
intuitive \textquotedblleft small size\textquotedblright\ (in general) of the
class $\delta_{p}^{n}$ (see \cite{bp, bpp, PAMS2} for details that justify
this intuition), in this section we look for results of the type%

\[
\Pi_{q}\circ%
{\displaystyle\bigcup\nolimits_{p\geq1}}
\delta_{p}^{n}\subset\Pi_{r;s}^{n}%
\]
for $q,r,s\in\lbrack1,\infty)$, i.e., for stronger results than those proposed
in the Problem \ref{opoi}. More precisely, for fixed Banach spaces
$X_{1},...,X_{n},Y,Z$ with certain properties we try to find $t,l,r,s$ so
that
\[
\Pi_{t;l}(Y;Z)\circ%
{\displaystyle\bigcup\nolimits_{p\geq1}}
\delta_{p}^{n}(X_{1},...,X_{n};Y)\subset\Pi_{r;s}^{n}(X_{1},...,X_{n};Z).
\]

In view of the important effect that cotype properties play in the theory of
absolutely summing operators (see for example \cite{michels, PellZ, dd, tt}),
in the next section we, in some sense, complement the results of \cite{popa}
by exploring the cotype of the Banach spaces involved.

If $Y=\mathbb{K}$ the following result gives an important and useful estimate
for the \textquotedblleft size\textquotedblright\ of the set of all
$p$-dominated scalar-valued multilinear operators:

\begin{theorem}
[Floret, Matos (1995) and P\'{e}rez-Garc\'{\i}a (2002)]\label{ff}If
$X_{1},...,X_{n}\,$\ are Banach spaces then
\[%
{\displaystyle\bigcup\nolimits_{p\geq1}}
\delta_{p}^{n}(X_{1},...,X_{n};\mathbb{K})\subset\Pi_{(1;2,...,2)}^{n}%
(X_{1},...,X_{n};\mathbb{K}).
\]

\end{theorem}

More precisely this result is due to Floret-Matos \cite{FM} for the complex
case and due to D. P\'{e}rez-Garc\'{\i}a \cite{Da} for the general case. It is
worth mentioning that, besides not explicitly mentioned, this result seems to
be essentially re-proved in \cite{popa}.

The following result is an application of the previous theorem:

\begin{proposition}
\label{hh}If $X_{1},...,X_{n},Y,Z$ are Banach spaces, then%
\[
\Pi_{s;1}(Y;Z)\circ%
{\displaystyle\bigcup\nolimits_{p\geq1}}
\delta_{p}^{n}(X_{1},...,X_{n};Y)\subset\Pi_{(s;2,...,2)}^{n}(X_{1}%
,...,X_{n};Z).
\]
for all $s\geq1.$
\end{proposition}

\begin{proof}
Let $T\in\Pi_{(s;1)}(Y,Z)$ and $R\in%
{\displaystyle\bigcup\nolimits_{p\geq1}}
\delta_{p}^{n}(X_{1},...,X_{n};Y).$ Consider $(x_{j}^{(k)})_{j=1}^{\infty}%
\in\ell_{2}^{w}\left(  X_{k}\right)  $ for all $k=1,...,n$. If $\varphi\in
Y^{\ast}$ from Theorem \ref{ff} we have
\[
\varphi\circ R\in%
{\displaystyle\bigcup\nolimits_{p\geq1}}
\delta_{p}^{n}(X_{1},...,X_{n};\mathbb{K})\subset\Pi_{(1;2,...,2)}^{n}%
(X_{1},...,X_{n};\mathbb{K}).
\]
Hence%
\[
\left(  \varphi\left(  R(x_{j}^{(1)},...,x_{j}^{(n)})\right)  \right)
_{j=1}^{\infty}\in\ell_{1}.
\]
We thus conclude that $\left(  R(x_{j}^{(1)},...,x_{j}^{(n)})\right)
_{j=1}^{\infty}\in\ell_{1}^{w}\left(  Y\right)  $ and so
\[
\left(  T\left(  R(x_{j}^{(1)},...,x_{j}^{(n)})\right)  \right)
_{j=1}^{\infty}\in\ell_{s}\left(  Z\right)  ,
\]
because $T\in\Pi_{s;1}(Y,Z).$
\end{proof}

When $X_{1}=\cdots=X_{n}$ are $\mathcal{L}_{\infty}$ spaces we have a quite
stronger result:

\begin{proposition}
If $Y,Z$ are Banach spaces and $X_{1}=\cdots=X_{n}$ are $\mathcal{L}_{\infty}$
spaces, then%
\[
\Pi_{(s;r)}(Y;Z)\circ\mathcal{L}(X_{1},...,X_{n};Y)\subset\Pi_{(s;2r,...,2r)}%
^{n}(X_{1},...,X_{n};Z)
\]
for all $s\geq r\geq1.$
\end{proposition}

\begin{proof}
Let $T\in\Pi_{(s;r)}(Y,Z)$ and $R\in\mathcal{L}(X_{1},...,X_{n};Y).$ Consider
$(x_{j}^{(k)})_{j=1}^{\infty}\in\ell_{2r}^{w}\left(  X_{k}\right)  $ for all
$k=1,...,n$. If $\varphi\in Y^{\ast}$ we have (from \cite[Theorem
3.15]{michels})
\[
\varphi\circ R\in\mathcal{L}(X_{1},...,X_{n};\mathbb{K})=\Pi_{(r;2r,...,2r)}%
^{n}(X_{1},...,X_{n};\mathbb{K}).
\]
Hence%
\[
\left(  \varphi\left(  R(x_{j}^{(1)},...,x_{j}^{(n)})\right)  \right)
_{j=1}^{\infty}\in\ell_{r}%
\]
and thus $\left(  R(x_{j}^{(1)},...,x_{j}^{(n)})\right)  _{j=1}^{\infty}%
\in\ell_{r}^{w}\left(  Y\right)  .$ Since $T\in\Pi_{(s;r)}(Y,Z),$ we conclude
that
\[
\left(  T\left(  R(x_{j}^{(1)},...,x_{j}^{(n)})\right)  \right)
_{j=1}^{\infty}\in\ell_{s}\left(  Z\right)  .
\]

\end{proof}

\section{Exploring the cotype of the spaces involved}

In this section we will need, as auxiliary results, some inclusions involving
cotype and absolutely summing multilinear operators. The following results can
be found in \cite[Theorem 3 and Remark 2]{jun} and \cite[Theorem 3.8]%
{michels}, by using complex interpolation, and \cite[Corollary 4.6]{popa3}:

\begin{theorem}
[Inclusion Theorem]\label{tt} Let $X_{1},...,X_{n}$ be Banach spaces with
cotype $s$ and $n\geq2$ be a positive integer:

(i) If $s=$ $2,$ then
\[
\Pi_{q;q}^{n}(X_{1},...,X_{n};Y)\subset\Pi_{p;p}^{n}(X_{1},...,X_{n};Y)
\]
holds true for $1\leq p\leq q\leq2$.

(ii) If $s>2,$ then
\[
\Pi_{q;q}^{n}(X_{1},...,X_{n};Y)\subset\Pi_{p;p}^{n}(X_{1},...,X_{n};Y)
\]
holds true for $1\leq p\leq q<s^{\ast}<2$.
\end{theorem}

As we will see in the next results, a far-reaching version (of independent
interest) of this theorem is valid. This result uses arguments from
\cite{blasco} and, in essence is contained in \cite{blasco}:

\begin{theorem}
\label{klmm}If $X_{1}$ has cotype $2$ and $1\leq p\leq s\leq2,$ then
\[
\Pi_{(s;s,t,...,t)}^{n}(X_{1},...,X_{n};Z)\subset\Pi_{(p;p,t,...,t)}^{n}%
(X_{1},...,X_{n};Z)
\]
for all $X_{2},...,X_{n}$,$Z$ and all $t\geq1.$ In particular%
\begin{equation}
\Pi_{(s;s,...,s)}^{n}(X_{1},...,X_{n};Z)\subset\Pi_{(p;p,s,...,s)}^{n}%
(X_{1},...,X_{n};Z)\subset\Pi_{(p;p,p,...,p)}^{n}(X_{1},...,X_{n};Z).
\label{uuu}%
\end{equation}

\end{theorem}

\begin{proof}
Since $X_{1}$ has cotype $2$, then, using results from \cite{blasco}, we have
\[
\ell_{p}^{w}(X_{1})=\ell_{r}\ell_{s}^{w}(X_{1})
\]
for
\[
\frac{1}{r}+\frac{1}{s}=\frac{1}{p}.
\]
Let $(x_{k}^{(1)})_{k=1}^{\infty}\in\ell_{p}^{w}(X_{1})$ and $(x_{k}%
^{(i)})_{k=1}^{\infty}\in\ell_{t}^{w}(X_{i})$ for $i=2,...,n$. So $x_{k}%
^{(1)}=\alpha_{k}y_{k},$ with $\left(  \alpha_{k}\right)  _{k=1}^{\infty}%
\in\ell_{r}$ and $\left(  y_{k}\right)  _{k=1}^{\infty}\in\ell_{s}$ for all
$k$. If $A\in\Pi_{(s;s,t,...,t)}^{n}(X_{1},...,X_{n};Z)$, then%
\begin{align*}
\left(
{\textstyle\sum\limits_{j=1}^{\infty}}
\left\Vert A(x_{j}^{(1)},...,x_{j}^{(n)}\right\Vert ^{p}\right)  ^{1/p}  &
=\left(
{\textstyle\sum\limits_{j=1}^{\infty}}
\left\Vert \alpha_{j}A(y_{j},x_{j}^{(2)},...,x_{j}^{(n)}\right\Vert
^{p}\right)  ^{1/p}\\
&  \leq\left(
{\textstyle\sum\limits_{j=1}^{\infty}}
\left\vert \alpha_{j}\right\vert ^{r}\right)  ^{1/r}\left(
{\textstyle\sum\limits_{j=1}^{\infty}}
\left\Vert A(y_{j},x_{j}^{(2)},...,x_{j}^{(n)}\right\Vert ^{s}\right)
^{1/s}<\infty
\end{align*}
and the proof is done.
\end{proof}

\begin{remark}
Note that using the inclusion theorem for absolutely summing multilinear
operators and Theorem \ref{klmm} we conclude that, in fact,%
\[
\Pi_{(s;s,...,s)}^{n}(X_{1},...,X_{n};Z)=\Pi_{(p;p,s,...,s)}^{n}%
(X_{1},...,X_{n};Z)
\]
under the hypotheses of Theorem \ref{klmm}.
\end{remark}

A similar result holds for spaces with cotype greater than $2$:

\begin{theorem}
\label{m2}If $X_{1}$ has cotype $s>2$ and $1\leq p\leq q<s^{\ast},$ then
\[
\Pi_{(q;q,t,...,t)}^{n}(X_{1},...,X_{n};Z)\subset\Pi_{(p;p,t,...,t)}^{n}%
(X_{1},...,X_{n};Z)
\]
for all $X_{2},...,X_{n}$,$Z$ and all $t\geq1.$ In particular%
\[
\Pi_{(q;q,...,q)}^{n}(X_{1},...,X_{n};Z)\subset\Pi_{(p;p,q,...,q)}^{n}%
(X_{1},...,X_{n};Z)\subset\Pi_{(p;p,p,...,p)}^{n}(X_{1},...,X_{n};Z).
\]

\end{theorem}

\begin{proof}
Since $X_{1}$ has cotype $s$, then%
\[
\ell_{p}^{w}(X_{1})=\ell_{r}\ell_{q}^{w}(X_{1})
\]
whenever $1\leq p\leq q<s^{\ast}$ with
\[
\frac{1}{r}+\frac{1}{q}=\frac{1}{p}%
\]
and the proof follows the lines of the proof of Theorem \ref{klmm}.
\end{proof}

\begin{remark}
\label{rem}Obviously, Theorems \ref{klmm} and \ref{m2} have an analogous
version when some $X_{j}$ (instead of necessarily $X_{1}$) has cotype $2.$
\end{remark}

The following result can be found in \cite[Theorem 3.10]{pv}:

\begin{proposition}
[P\'{e}rez-Garc\'{\i}a and Villanueva, 2003]\label{kl}If $Y$ has cotype finite
cotype $s$, then%
\[%
{\displaystyle\bigcup\nolimits_{p\geq1}}
\delta_{p}^{n}(X_{1},...,X_{n};Y)\subset\Pi_{(s;2,...,2)}^{\text{mult,}%
n}(X_{1},...,X_{n};Y)\subset\Pi_{(s;2,...,2)}^{n}(X_{1},...,X_{n};Z)
\]
for all Banach spaces $X_{1},...,X_{n},Z.$
\end{proposition}

In particular, the previous result shows that%
\[
\mathcal{L}(Y;Z)\circ%
{\displaystyle\bigcup\nolimits_{p\geq1}}
\delta_{p}^{n}(X_{1},...,X_{n};Y)\subset\Pi_{(s;2,...,2)}^{n}(X_{1}%
,...,X_{n};Z)
\]
for all Banach spaces $X_{1},...,X_{n},Z$ and $Y$ with finite cotype $s$.

We will focus our attention in the case $s=2$ of Proposition \ref{kl}. It is
well-known that if $X_{1},...,X_{n},Y$ have cotype $2$, then%
\[
\Pi_{2}^{\text{mult,}n}(X_{1},...,X_{n};Y)=\Pi_{r}^{\text{mult,}n}%
(X_{1},...,X_{n};Y)
\]
for every $1\leq r\leq2$. Hence

\begin{corollary}
If $X_{1},...,X_{n},Y$ have cotype $2$, then%
\[
\Pi_{q}(Y;Z)\circ%
{\displaystyle\bigcup\nolimits_{p\geq1}}
\delta_{p}^{n}(X_{1},...,X_{n};Y)\subset\Pi_{r}^{\text{mult,}n}(X_{1}%
,...,X_{n};Y)
\]
for all $q\in\lbrack1,\infty),1\leq r\leq2$ and all Banach space $Z$.
\end{corollary}

Under the assumptions of Proposition \ref{kl} in general $\Pi_{1}^{n}%
(X_{1},...,X_{n};Z)$ is not contained in $\Pi_{2}^{n}(X_{1},...,X_{n};Z)$. The
map $T:\ell_{2}\times\ell_{2}\rightarrow\ell_{1}$ given by $T(x,y)=(x_{j}%
y_{j})_{j=1}^{\infty}$ belongs to $\Pi_{1}^{2}\left(  \ell_{2},\ell_{2}%
;\ell_{1}\right)  $ but not to $\Pi_{2}^{2}\left(  \ell_{2},\ell_{2};\ell
_{1}\right)  .$ In fact Theorem \ref{klmm} (and Remark \ref{rem}), in
particular, assures that if some $X_{j}$ has cotype $2$, then%
\[
\Pi_{(2;2,2,...,2)}^{n}(X_{1},...,X_{n};Z)\subset\Pi_{(p;p,p,...,p)}^{n}%
(X_{1},...,X_{n};Z)
\]
for all $1\leq p\leq2.$ So we have:

\begin{corollary}
\label{lop} If $X_{j}$ has cotype $2$ for some $j$ and $Y$ has cotype $2$ then%
\[
\Pi_{q}(Y;Z)\circ%
{\displaystyle\bigcup\nolimits_{p\geq1}}
\delta_{p}^{n}(X_{1},...,X_{n};Y)\subset\Pi_{r}^{n}(X_{1},...,X_{n};Z)
\]
for all $q\in\lbrack1,\infty),$ all $r\in\lbrack1,2]$ and all Banach spaces
$X_{1},...,X_{j-1},X_{j+1},...,X_{n},Z.$
\end{corollary}

Hence Corollary \ref{lop} is quite stronger from Theorem \ref{op} for this
special case where $X_{j}$ has cotype $2$ for some $j$ and $Y$ has cotype $2$.

Now we explore some consequences of Proposition \ref{hh}. Note that if $Y$ and
$Z$ have cotype $2,$ for example, it is well-known that $\Pi_{1}(Y,Z)=\Pi
_{q}(Y,Z)$ for all $1\leq q<\infty$ (see \cite[Corollary 11.16]{djt}) and we get:

\begin{corollary}
\label{opp}If $Y$ and $Z$ have cotype $2,$ then%
\[
\Pi_{q}(Y;Z)\circ%
{\displaystyle\bigcup\nolimits_{p\geq1}}
\delta_{p}^{n}(X_{1},...,X_{n};Y)\subset\Pi_{(1;2,...,2)}^{n}(X_{1}%
,...,X_{n};Z)\subset\Pi_{(2;2,...,2)}^{n}(X_{1},...,X_{n};Z)
\]
for all $q\in\lbrack1,\infty)$ and all Banach spaces $X_{1},...,X_{n}.$
\end{corollary}

Also, using \cite[Corollary 11.16]{djt} and Proposition \ref{hh} we have:

\begin{corollary}
If $Y$ has cotype $2,$ then%
\[
\Pi_{q}(Y;Z)\circ%
{\displaystyle\bigcup\nolimits_{p\geq1}}
\delta_{p}^{n}(X_{1},...,X_{n};Y)\subset\Pi_{(1;2,...,2)}^{n}(X_{1}%
,...,X_{n};Z)\subset\Pi_{(2;2,...,2)}^{n}(X_{1},...,X_{n};Z)
\]
for all $q\in\lbrack1,2]$ and all Banach spaces $X_{1},...,X_{n},Z.$
\end{corollary}

\section{A partial solution to Problem \ref{novo}}

In this last section we present some very simple but apparently useful and
overlooked remarks on Problem \ref{novo}. It is easy to see that
\begin{equation}
\frac{p}{n}\leq r\leq p\Rightarrow\delta_{p}^{n}\subset\Pi_{r}^{n}. \label{nj}%
\end{equation}

We thus have:

\begin{proposition}
Let $p,q,r\in(1,\infty)$ be such that $r\leq p$. Then the inclusion%
\begin{equation}
\Pi_{q}\circ\delta_{p}^{n}\subset\Pi_{r}^{n} \label{ypp}%
\end{equation}
is valid for all $n\geq\frac{p}{r}.$
\end{proposition}

So, \textit{a fortiori}, we have a partial answer to Problem \ref{novo} (note
that we do not actually need the hypothesis $\frac{1}{p}+\frac{1}{q}=\frac
{1}{r}$):

\begin{corollary}
Let $p,q,r\in(1,\infty)$ be such that $\frac{1}{r}=\frac{1}{p}+\frac{1}{q}$.
Then the inclusion
\begin{equation}
\Pi_{q}\circ\delta_{p}^{n}\subset\Pi_{r}^{n} \label{ptt}%
\end{equation}
is valid for all $n\geq\frac{p}{r}.$
\end{corollary}

\end{document}